\documentclass[]{article}
\pdfoutput=1
\usepackage{graphicx}
\usepackage{amsmath}
\usepackage{amsthm}

\newtheorem{thm}{Theorem}[section]

\newtheorem{prop}[thm]{Proposition}
\newcommand{\epn}{\ensuremath{\mathrm{epn}}}
\newcommand{\ipn}{\ensuremath{\mathrm{ipn}}}

\begin{document}

\title{Secure Total Domination Number in Maximal Outerplanar Graphs}

\author{Yasufumi Aita \and Toru Araki}
\date{Gunma University}
\maketitle

\begin{abstract}
  A subset $S$ of vertices in a graph $G$ is a secure total dominating
  set of $G$ if $S$ is a total dominating set of $G$ and, for each
  vertex $u \not\in S$, there is a vertex $v \in S$ such that $uv$ is
  an edge and $(S \setminus \{v\}) \cup \{u\}$ is also a total
  dominating set of $G$.
  We show that if $G$ is a maximal outerplanar graph of order $n$,
  then $G$ has a total secure dominating set of size at most
  $\lfloor 2n/3 \rfloor$.
  Moreover, if an outerplanar graph $G$ of order $n$, then each
  secure total dominating set has at least $\lceil (n+2)/3 \rceil$
  vertices.
  We show that these bounds are best possible.

  \vspace{1em} Secure total domination, total domination, maximal
  outerplanar graphs, upper bound, lower bound.
\end{abstract}

\section{Introduction}
\label{sec:introduction}

We consider finite undirected graph $G$ with vertex set $V(G)$ and
edge set $E(G)$ without self-loops.
The \emph{open neighborhood} of a vertex $v \in V(G)$ is defined by
$N_{G}(v) = \{u \mid vu \in E(G) \}$, and the \emph{closed
  neighborhood} of $v$ is $N_{G}[v]=N_{G}(u) \cup \{v\}$.
We denote by $\deg_{G} v=|N_{G}(v)|$ the degree of $v$.
For a subset $U \subseteq V(G)$, the subgraph induced by $U$ is
denoted by $G[U]$.
For a proper subset $U \subset V(G)$, we denote by $G-U$ the graph
obtained by removing vertices in $U$ and their incident edges from
$G$.
A vertex $x$ is a \emph{cut-vertex} of a connected graph $G$ if
$G-\{v\}$ is disconnected.
A \emph{block} of a connected graph is a maximal subgraph
that has no cut-vertices.
Any terminology not defined here, we refer to~\cite{chartrand11}.

% A vertex $v$ is said to \emph{dominate} itself and each vertex in
% $N_{G}(v)$, that is, $v$ dominates the vertices in $N_{G}[v]$.
A set $S \subseteq V(G)$ is a \emph{dominating set} of $G$ if each
vertex $u \in V(G) \setminus S$ is adjacent to some vertex in $S$.
The \emph{domination number} $\gamma(G)$ of $G$ is the smallest
cardinality of a dominating set of $G$.
A set $S \subseteq V(G)$ is a \emph{total dominating set} of $G$ if
every vertex is adjacent to some vertex in $S$.
The \emph{total domination number} $\gamma_{t}(G)$ of $G$ is the smallest
cardinality of a total dominating set of $G$.

% Let $S$ be a dominating set of $G$.
% A vertex $v \in S$ is said to \emph{defend} $u \not\in S$
% if $uv \in E(G)$ and $S' =(S \setminus \{v\}) \cup \{u\}$ is also a
% dominating set.
Recently, secure domination and secure total domination were
introduced in~\cite{benecke07:_secur,cockayne08:_protec}.
A dominating set $S$ is a \emph{secure dominating set} (or an
\emph{SDS}) if, for any $u \in V(G) \setminus S$, there exists a
vertex $v \in S$ such that $uv \in E(G)$ and $(S \setminus \{v\}) \cup
\{u\}$ is also a dominating set.
The \emph{secure domination number} $\gamma_{s}(G)$ of $G$ is the
smallest cardinality of a secure dominating set of $G$.
A total dominating set $S$ is a \emph{secure total dominating set} (or
an \emph{STDS}) if, for every $u \in V(G) \setminus S$, there exists a
vertex $v \in S$ such that $uv \in E(G)$ and $(S \setminus \{v\}) \cup
\{u\}$ is also a total dominating set.
The \emph{secure total domination number} $\gamma_{st}(G)$ of $G$ is the
smallest cardinality of a secure total dominating set of $G$.

% The notion of secure domination has been researched extensively.
% Cockayne et al.~\cite{cockayne08:_protec} investigated some
% fundamental properties of an SDS, and obtained exact values of
% $\gamma_{s}(G)$ for some graph classes, such as paths, cycles,
% complete multipartite graphs.

Various aspects of secure domination and secure total
domination have been researched~\cite{burger16,merouane15,klostermeyer08:_secur,burger08:_vertex,li17,martinez19}.
The secure total domination problem is NP-hard even when restricted to
chordal bipartite and split graphs~\cite{duginov17:_secur}.
Some polynomial-time algorithms for computing the secure domination
number of some restricted graph classes are investigated
~\cite{burger14,araki18:_secur,pradhan17,araki19:_secur,jha19,zou19}.

A graph $G$ is \emph{outerplanar} if it has a crossing-free embedding
in the plane such that all vertices belong to the boundary of its
outer face (the unbounded face).
A \emph{maximal outerplanar graph} (or just a \emph{mop}) is an
outerplanar graph such that the addition of a single edge results in a
graph that is not outerplanar.
Matheson and Tarjan~\cite{matheson96:_domin} proved a tight upper
bound for the domination number on the class of \emph{triangulated
  discs}: graphs that have an embedding in the plane such that all of
their faces are triangles, except possibly one.
They proved that $\gamma(G) \leq n/3$ for any $n$-vertex triangulated
disc.
For maximal outerplanar graphs, better upper bounds are obtained.
Campos and Wakabayashi~\cite{campos13} showed that if $G$ is a mop of
$n$ vertices, then $\gamma(G) \leq (n+k)/4$ where $k$ is the number of
vertices of degree 2.
Tokunaga proved the same result independently
in~\cite{tokunaga13:_domin}.
Li et al.~improved the result by showing that $\gamma(G)
\leq (n+t)/4$, where $t$ is the number of pairs of consecutive
degree 2 vertices with distance at least 3 on the outer
cycle~\cite{li16}.
Dorfling et al.~\cite{dorfling16:_total2} proved that $\gamma_{t}(G)
\leq \lfloor 2n/5 \rfloor$, and then Lema\'{n}ska et
al.~\cite{lemanska17:_total} gave an alternative proof of it.
The second author~\cite{araki18} proved that $\gamma_{s}(G) \leq
\lceil 3n/7 \rceil$ and the upper bound is sharp.

In this paper, we give sharp upper and lower bounds for the secure
total domination number of maximal outerplanar graphs.
We prove that, for any maximal outerplanar graph $G$ of $n \geq 3$
vertices, $\lceil (n+2)/3 \rceil \leq \gamma_{st}(G) \leq \lfloor 2n/3
\rfloor$, and these bounds are sharp.
In Section~\ref{sec:structure-mop}, some properties of maximal
outerplanar graphs are described.
Then, in Section~\ref{sec:upperbound} and~\ref{sec:lower-bound}, we
establish upper and lower bounds for the secure total domination
number.


\section{Structure of maximal outerplanar graphs}
\label{sec:structure-mop}

In this section, we give some properties of maximal outerplanar
graphs.
O'Rourke~\cite{orourke87} pointed out that every mop has a unique
Hamiltonian cycle.
Thus, the Hamiltonian cycle is the boundary of the mop.
The Hamiltonian cycle of a mop $G$ is denoted by $C(G)$.

\begin{prop}[\cite{li16}]
  \label{prop:deg2}
  Every maximal outerplanar graph contains at least two vertices of
  degree 2.
\end{prop}

\begin{prop}[\cite{orourke87}]
  \label{thm:subgraphs}
  Let $G$ be a mop of $n \geq 5$ vertices.
  There are consecutive vertices in $C(G)$ that induce at least one
  of the following eight subgraphs which are illustrated in
  Fig~\ref{fig:eight_subgraphs}:
  \begin{itemize}
  \item[(a)] Subgraph induced by five consecutive vertices $u,v,w,x,y$
    such that $vx,ux,uy \in E(G)$ or the mirror image of it.
  \item[(b)] Subgraph induced by five consecutive vertices $u,v,w,x,y$
    such that $uw,ux,uy \in E(G)$ or the mirror image of it.
  \item[(c)] Subgraph induced by five consecutive vertices $u,v,w,x,y$
    such that $uw,wy,uy \in E(G)$.
  \item[(d)] Subgraph induced by six consecutive vertices
    $t,u,v,w,x,y$ such that $tw,uw,wy,ty \in E(G)$ or the mirror image
    of it.
  \item[(e)] Subgraph induced by six consecutive vertices
    $t,u,v,w,x,y$ such that $tv,tw,wy,ty \in E(G)$ or the mirror image
    of it.
  \item[(f)] Subgraph induced by seven consecutive vertices
    $t,u,v,w,x,y,z$ such that $tw,uw,wy,wz,tz \in E(G)$.
  \item[(g)] Subgraph induced by seven consecutive vertices
    $t,u,v,w,x,y,z$ such that $tw,uw,wz,xz,tz \in E(G)$ or the mirror
    image of it.
  \item[(h)] Subgraph induced by seven consecutive vertices
    $t,u,v,w,x,y,z$ such that $tv,tw,wz,xz,tz \in E(G)$.
  \end{itemize}
\end{prop}

\begin{figure}[tb]
  \centering
  \includegraphics[width=\textwidth]{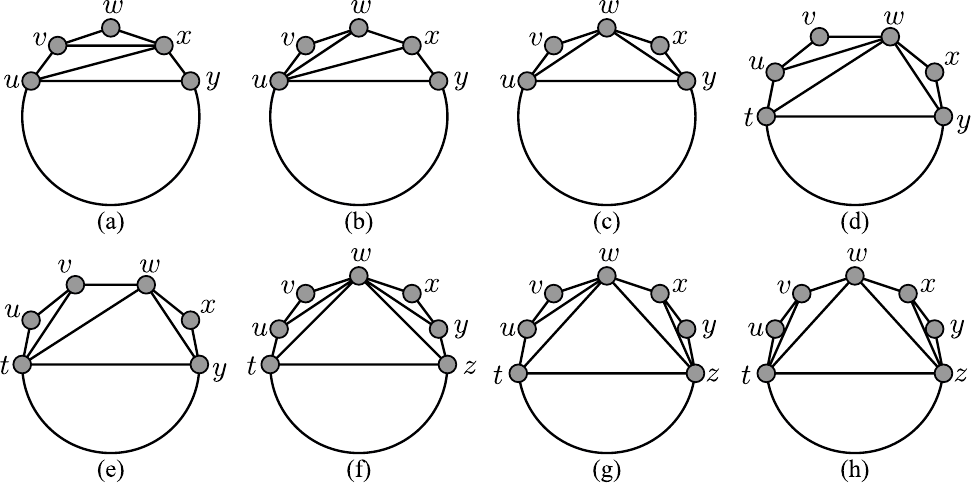}
  \caption{Eight induced subgraphs in a maximal outerplanar graph.}
  \label{fig:eight_subgraphs}
\end{figure}

% For example, in the mop of Fig.~\ref{fig:example_mop},
% \begin{itemize}
% \item vertices $a,b,c,d,e$ induce the subgraph isomorphic to
%   Fig.~\ref{fig:eight_subgraphs}(a),
% \item vertices $l,m,n,o,p$ induce the subgraph isomorphic to
%   Fig.~\ref{fig:eight_subgraphs}(b),
% \item vertices $g,h,i,j,k,l$ induce the subgraph isomorphic to
%   (the mirror image of) Fig.~\ref{fig:eight_subgraphs}(d).
% \end{itemize}

% \begin{figure}[tb]
%   \centering
%   \includegraphics{example_outerplanar.pdf}
%   \caption{An example of an outerplanar graph.}
%   \label{fig:example_mop}
% \end{figure}

A graph $H$ is a \emph{subdivision} of $G$ if either $G=H$ or $H$ can
be obtained from $G$ by inserting vertices of degree 2 into the edges
of $G$.
The following characterization of outerplanar graphs is known
(see~\cite{chartrand11}, for example).

\begin{thm}
  \label{thm:charact}
  A graph $G$ is outerplanar if and only if $G$ contains no subgraph
  that is a subdivision of the complete graph $K_{4}$ or the complete
  bipartite graph $K_{2,3}$.
\end{thm}


\section{Upper bound}
\label{sec:upperbound}

In this section, we show the next theorem.

\begin{thm}
  \label{thm:upper_bound}
  For any maximal outerplanar graph $G$ of $n \geq 3$ vertices,
  \begin{displaymath}
    \gamma_{st}(G) \leq \left\lfloor \frac{2n}{3} \right\rfloor.
  \end{displaymath}
\end{thm}

Let $S$ be a total dominating set of $G$.
For $u \notin S$ and $v \in S$, if $uv \in E(G)$ and
$(S \setminus \{v\}) \cup \{u\}$ is also a total dominating set of
$G$, we say that \emph{$v$ totally $S$-defends $u$}.
Hence $S$ is an STDS of $G$ if and only if, for any vertex $u \notin
S$, there exists a vertex $v \in S$ such that $v$ totally $S$-defends
$u$.

For a set $S \subseteq V(G)$, the \emph{external private neighborhood
  of $v \in S$ with respect to $S$} is defined by
\begin{displaymath}
  \epn(v,S) = \{w \mid w \notin S \text{ and } N(w) \cap S = \{v\}\}.
\end{displaymath}
The \emph{internal private neighborhood of $v \in S$ with respect to
  $S$} is defined by
\begin{displaymath}
  \ipn(v,S) = \{u \mid u \in S \text{ and } N(u) \cap S = \{v\}\}.
\end{displaymath}
The next proposition was given in~\cite{klostermeyer08:_secur}.

\begin{prop}[\cite{klostermeyer08:_secur}]
  \label{prop:stds_ch}
  Let $S$ be a total dominating set of $G$.
  A vertex $v \in S$ totally $S$-defends $u \notin S$ if and only if
  (1) $\epn(v,S) = \emptyset$, and (2) $\{v\} \cup \ipn(v,S) \subseteq
  N(u)$.
\end{prop}

% \begin{thm}
%   \label{prop:minstds}
%   For a maximal outerplanar graph $G$ of $n \geq 4$ vertices, there is
%   a $\gamma_{st}$-set $S$ such that $S$ has no vertex of degree 2 and
%   the two neighbors of the vertex of degree 2 are members in $S$.
% \end{thm}
% \begin{proof}
%   Assume that $S$ is a $\gamma_{st}$-set of a mop $G$, and also assume
%   that $\deg_{G}(v) = 2$ for some $v \in S$.
%   Let $N_{G}(v)=\{x,y\}$.
%   Since $S$ is a total dominating set, at least one of $x$ or $y$ is
%   in $S$.

%   We first assume that $x \in S$ and $y \notin S$.
%   In this case, clearly $S' = (S \setminus \{v\}) \cup \{y\}$ is an
%   STDS of $G$.

%   Next we assume that $x,y \in S$.
%   Assume that either $\ipn(x,S) = \emptyset$ or $\ipn(y,S) =
%   \emptyset$.
%   Without loss of generality, we may assume that $\ipn(x,S) =
%   \emptyset$.
%   Let $S' = S \setminus \{v\}$.
%   In this case, by Proposition~\ref{prop:stds_ch}, $x$ totally
%   $S$-defends $v$.
%   Hence $S'$ is a secure total dominating set of $G$, but this
%   contradicts the fact $S$ is a $\gamma_{st}$-set.
%   So $\ipn(x,S) \neq\emptyset$ and  $\ipn(y,S) \neq \emptyset$.
% \end{proof}

We prove the following theorem instead of Theorem~\ref{thm:upper_bound}.

\begin{thm}
  \label{thm:upper_bound2}
  Let $G$ be any maximal outerplanar graph $G$ of $n \geq 4$ vertices.
  Then, there exists a secure total dominating set $S$ of $G$ such
  that $|S| \leq \lfloor 2n/3 \rfloor$ and $S$ contains no vertex of
  degree 2.
\end{thm}
\begin{proof}
  It should be noted that, if an STDS $S$ does not contain a vertex
  $u$ of degree~2, then the two vertices adjacent to $u$ are in $S$
  since, for any vertex $x \in S$, we have $\epn(x,S)=\emptyset$ by
  Proposition~\ref{prop:stds_ch}.

  We prove this theorem by induction on $n$.
  For $4 \leq n \leq 6$, the proposition is true as illustrated in
  Fig~\ref{fig:minimumSTDS}.
  Each of the STDSs has at most $\lfloor 2n/3 \rfloor$ vertices and
  does not have a vertex of degree~2.

  \begin{figure}[tb]
    \centering
    \includegraphics[width=\textwidth]{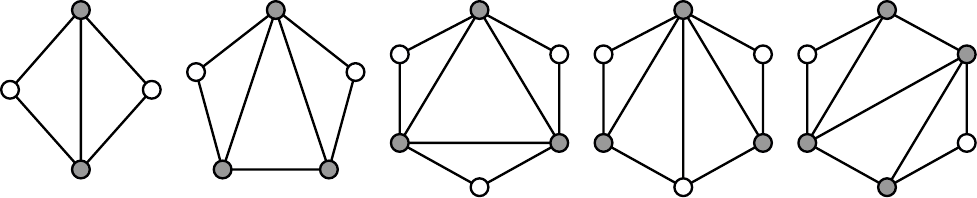}
    \caption{Minimum STDSs for maximal outerplanar graphs of
      $n \leq 6$ vertices.  The gray vertices are members of STDSs.}
    \label{fig:minimumSTDS}
  \end{figure}

  Let $G$ be a maximal outerplanar graph of $n \geq 7$ vertices.
  We assume that, for any $4 \leq n' \leq n-1$, the theorem is
  % true for any mop of $n'$ vertices.
  true for any mop of order $n'$.
  By Proposition~\ref{thm:subgraphs}, $G$ has a subgraph $H$ which is
  induced by consecutive vertices in $C(G)$ illustrated in
  Fig.~\ref{fig:eight_subgraphs}.
  Consider the following eight cases.

  \vspace{1em}
  \noindent
  \textbf{(Case a)} $H$ is Fig.~\ref{fig:eight_subgraphs}(a).
  Let $G' = G-\{v,w,x\}$.
  Then $G'$ is a mop of order $n-3$.
  By the induction hypothesis, $G'$ has an STDS $S'$ such that
  $|S'| \leq \lfloor 2(n-3)/3 \rfloor$.
  Then $S = S' \cup \{v,x\}$ is an STDS of $G$ since $v$ totally
  $S$-defends $w$, and
  $|S| = |S'|+2 \leq \lfloor 2(n-3)/3 \rfloor + 2 = \lfloor 2n/3
  \rfloor$.

  \vspace{1em}
  \noindent
  \textbf{(Case b)} $H$ is Fig.~\ref{fig:eight_subgraphs}(b).
  Let $G' = G-\{v,w,x\}$.
  Then $G'$ is a mop of $n-3$ vertices.
  By the induction hypothesis, $G'$ has an STDS $S'$ such that $|S'|
  \leq \lfloor 2(n-3)/3 \rfloor$.
  Consider the next two cases.

  \noindent
  (b-1) First we assume that $u \in S'$.
  In this case, let $S = S' \cup \{w,x\}$.
  Then $S$ is a total dominating set of $G$, and $w$ totally
  $S$-defends $v$, and $x$ totally $S$-defends $y$ when $y \notin S$.
  Hence $S$ is an STDS of $G$ and
  $|S| = |S'|+2 \leq \lfloor 2(n-3)/3 \rfloor + 2 = \lfloor 2n/3 \rfloor$.

  \noindent
  (b-2) Next we assume that $u \notin S'$.
  Then there is either $y \in S'$ or $y$ is totally $S'$-defended by some
  vertex $r \in S'$.
  In this case, let $S = S' \cup \{u,w\}$.
  Then $S$ is a total dominating set of $G$, and $w$ totally
  $S$-defends $v$ and $x$, and $y$ is totally $S$-defended by $r$.
  Hence $S$ is an STDS of $G$ and
  $|S| = |S'|+2 \leq \lfloor 2(n-3)/3 \rfloor + 2 = \lfloor 2n/3 \rfloor$.

  \vspace{1em}
  \noindent
  \textbf{(Case c)} $H$ is Fig.~\ref{fig:eight_subgraphs}(c).
  Let $G' = G-\{v,x\}$.
  Then $G'$ is a mop of $n-2$ vertices.
  By the induction hypothesis, $G'$ has an STDS $S'$ such that $|S'|
  \leq \lfloor 2(n-2)/3 \rfloor$ and $w \notin S'$ and $u,y \in S'$.
  Let $S = S' \cup \{w\}$.
  Then $S$ is an STDS of $G$ since $w$ totally $S$-defends $v$ and
  $x$, and $|S| = |S'|+1 \leq \lfloor 2(n-2)/3 \rfloor + 1 \leq \lfloor
  2n/3 \rfloor$.

  \vspace{1em}
  \noindent
  \textbf{(Case d)} $H$ is Fig.~\ref{fig:eight_subgraphs}(d).
  Let $G' = G-\{u,v,x\}$.
  Then $G'$ is a mop of $n-3$ vertices.
  By the induction hypothesis, $G'$ has an STDS $S'$ such that $|S'|
  \leq \lfloor 2(n-3)/3 \rfloor$ and $w \notin S'$ and $t,y \in S'$.
  Let $S = S' \cup \{u,w\}$.
  Then $S$ is an STDS of $G$ since $w$ totally $S$-defends $v$ and
  $x$, and $|S| = |S'|+2 \leq \lfloor 2(n-3)/3 \rfloor + 2 = \lfloor
  2n/3 \rfloor$.

  \vspace{1em}
  \noindent
  \textbf{(Case e)} $H$ is Fig.~\ref{fig:eight_subgraphs}(e).
  Let $G' = G-\{u,v,x\}$.
  Then $G'$ is a mop of $n-3$ vertices.
  By the induction hypothesis, $G'$ has an STDS $S'$ such that $|S'|
  \leq \lfloor 2(n-3)/3 \rfloor$ and $w \notin S'$ and $t,y \in S'$.
  Let $S = S' \cup \{v,w\}$.
  Then $S$ is an STDS of $G$ since $v$ and $w$ totally $S$-defend $u$
  and $x$, respectively, and $|S| = |S'|+2 \leq \lfloor 2(n-3)/3
  \rfloor + 2 = \lfloor 2n/3 \rfloor$.

  \vspace{1em}
  \noindent
  \textbf{(Case f)} $H$ is Fig.~\ref{fig:eight_subgraphs}(f).
  Let $G' = G-\{u,v,x\}$.
  Then $G'$ is a mop of $n-3$ vertices.
  By the induction hypothesis, $G'$ has an STDS $S'$ such that $|S'|
  \leq \lfloor 2(n-3)/3 \rfloor$ and $y \notin S'$ and $w,z \in S'$.
  Let $S = S' \cup \{u,y\}$.
  Then $S$ is an STDS of $G$ since $w$ totally $S$-defends $v,x$ and
  $t$, and $|S| = |S'|+2 \leq \lfloor 2(n-3)/3 \rfloor + 2 = \lfloor
  2n/3 \rfloor$.

  \vspace{1em}
  \noindent
  \textbf{(Case g)} $H$ is Fig.~\ref{fig:eight_subgraphs}(g).
  Let $G' = G-\{u,v,y\}$.
  Then $G'$ is a mop of $n-3$ vertices.
  By the induction hypothesis, $G'$ has an STDS $S'$ such that $|S'|
  \leq \lfloor 2(n-3)/3 \rfloor$ and $x \notin S$ and $w,z \in S$.
  Let $S = S' \cup \{u,x\}$.
  Then $S$ is an STDS of $G$ since $w$ totally $S$-defends $v$ and
  $t$, and $x$ totally $S$-defends $y$.
  We obtain
  $|S| = |S'|+2 \leq \lfloor 2(n-3)/3 \rfloor + 2 = \lfloor 2n/3
  \rfloor$.

  \vspace{1em}
  \noindent
  \textbf{(Case h)} $H$ is Fig.~\ref{fig:eight_subgraphs}(h).
  Let $G' = G-\{u,v,y\}$.
  Then $G'$ is a mop of $n-3$ vertices.
  By the induction hypothesis, $G'$ has an STDS $S'$ such that $|S'|
  \leq \lfloor 2(n-3)/3 \rfloor$ and $x \notin S'$ and $w,z \in S'$.
  Consider two cases below.

  \noindent
  (h-1) First assume that $t \in S'$.
  In this case, let $S = S' \cup \{v,x\}$.
  Then $S$ is an STDS of $G$ since $v$ totally $S$-defends $u$, and
  $x$ totally $S$-defends $y$, and $|S| = |S'|+2 \leq \lfloor
  2(n-3)/3  \rfloor + 2 = \lfloor 2n/3 \rfloor$.

  \noindent
  (h-2) Next we assume that $t \notin S'$.
  In this case, $w \in S'$ is contained in $\ipn(z, S')$.
  Hence, by Proposition~\ref{prop:stds_ch}, for any vertex $s \in (V
  \setminus S') \setminus \{t,x\}$, $z$ does not totally $S'$-defend
  $s$.
  Thus the vertex $s$ is totally $S'$-defended by some vertex other
  than $z$.
  Let $S = (S' \setminus \{w\}) \cup \{t,v,x\}$.
  Then $S$ is a total dominating set of $G$, and $x$ totally
  $S$-defends $y$ and $w$, and $v$ totally $S$-defends $u$.
  We have $|S| = (|S'|-1)+3 \leq \lfloor 2(n-3)/3  \rfloor + 2 =
  \lfloor 2n/3 \rfloor$.

  \vspace{1em}
  For each of the above eight cases, we obtain desired STDSs of $G$.
  Hence we complete the proof.
\end{proof}

If $n=3$, then clearly we obtain $\gamma_{st}(G)=2$.
Hence, Theorem~\ref{thm:upper_bound} follows from
Theorem~\ref{thm:upper_bound2}.
The upper bound in Theorem~\ref{thm:upper_bound} is best possible.
For any $k \geq 1$, let $H_{k}$ be a graph as follows.
\begin{align*}
  V(H_{k}) &= \{a_{1},a_{2},\dots,a_{k}\} \cup
             \{b_{1},b_{2},\dots,b_{k}\} \cup
             \{c_{1},c_{2},\dots,c_{k}\}, \\
  E(H_{k}) &= \{a_{i}b_{i}, b_{i}c_{i}, c_{i}a_{i} \mid i = 1,2,\dots,k\}
             \cup \{c_{i}a_{i+1} \mid i = 1,2,\dots,k-1\} \\
           &{} \cup \{a_{1}a_{i}, a_{1}c_{i} \mid i=2,3,\dots,k\}.
\end{align*}
$H_{k}$ has a mop of $n=3k$ vertices, and $\deg_{H_{k}}(b_{i}) = 2$ for
each $i=1,2,\dots,k$.
For example, $H_{4}$ is represented in Fig.~\ref{fig:upper}.
We can see that the set
$S=\{a_{1},c_{1},a_{2},c_{2},\dots,a_{k},c_{k}\}$ is an STDS of
$H_{k}$ and $|S|=\lfloor 2n/3 \rfloor = 2k$.
We show that $S$ is a minimum STDS of $H_{k}$.

\begin{figure}[tb]
  \centering
  \includegraphics{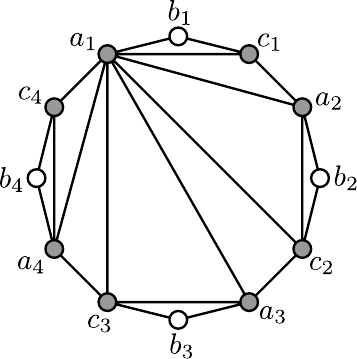}
  \caption{A maximal outerplanar graph $H_{k}$ when $k=4$.  The set of
    the gray vertices form a minimum STDS of $H_{k}$.}
  \label{fig:upper}
\end{figure}

Assume that there is an STDS $S'$ of $2k-1$ vertices.
Let $T_{i}=\{a_{i},b_{i},c_{i}\}$ for $k=1,2,\dots,k$.
By the pigeonhole principle, there is a subset $T_{j}$ such that
$|S' \cap T_{j}| \leq 1$.
If $S' \cap T_{j} = \emptyset$, then $S'$ is not a dominating set.
If $S' \cap T_{j} = \{b_{j}\}$, then $S'$ is not a total dominating
set since both of $a_{j}$ and $c_{j}$ are not in $S'$.
Finally, consider $S' \cap T_{j} = \{a_{j}\}$ or $\{c_{j}\}$.
Without loss of generality, we assume $S' \cap T_{j} = \{a_{j}\}$.
Since $b_{j},c_{j} \notin S'$, $b_{j} \in \epn(a_{j},S')$.
By Proposition~\ref{prop:stds_ch}, $a_{j}$ cannot totally $S'$-defend
$b_{j}$, and hence $S'$ is not an STDS.

From the above discussion, $S$ is an STDS of $G$, and
$\gamma_{st}(H_{k})=2k$.

The graph $H_{k}$ is order $3k$ and has $k$ vertices of degree 2.
The authors could not find an mop $G$ of order $n$ with
$\gamma_{st}(G)=\lfloor 2n/3 \rfloor$ and $k < \lfloor n/3 \rfloor$,
where $k$ is the number of vertices of degree 2.


\section{Lower bound}
\label{sec:lower-bound}

The purpose of this section is to show the next theorem.

\begin{thm}
  \label{thm:lower_bound}
  For an outerplanar graph $G$ of $n \geq 3$ vertices,
  \begin{displaymath}
    \gamma_{st}(G) \geq \left\lceil \frac{n+2}{3} \right\rceil.
  \end{displaymath}
\end{thm}

A set of vertices $S \subseteq V(G)$ is a \emph{2-dominating set} of
$G$ if $|S \cap N_{G}(u)| \geq 2$ for any $u \notin S$.
The \emph{2-domination number} of $G$, denoted by $\gamma_{2}(G)$, is
the minimum cardinality of a 2-dominating set of $G$.
In this section, we show that
$\gamma_{2}(G) \geq \lceil (n+2)/3 \rceil$ for an outerplanar
graph $G$ of $n$ vertices.
By Theorem~\ref{prop:stds_ch}, if $S$ is an STDS of $G$, we have
$\epn(v,S)=\emptyset$ for any $v \in S$, and thus $S$ is a
2-dominating set.
This means that $\gamma_{2}(G) \leq \gamma_{st}(G)$.
Hence, if $\gamma_{2}(G) \geq \lceil (n+2)/3 \rceil$, we obtain
Theorem~\ref{thm:lower_bound}.

Let $S$ be a 2-dominating set of $G$ and $T=V(G) \setminus S$,
and let $|S|=x$ and $|T| = y$.
Thus $x + y = n$.
If $y \leq 2x-2$, we obtain $x \geq \lceil (n+2)/3 \rceil$.

In order to show $y \leq 2x-2$, we suppose to the contrary that $y \geq
2x-1$.
Let $B(S)$ be a spanning subgraph of $G$ such that
\begin{itemize}
\item $B(S)$ is a bipartite graph with the bipartition $S \cup T$, and
\item each $v \in T$ is adjacent to exactly two vertices in $S$.
\end{itemize}
Since $S$ is a 2-dominating set, $v \notin S$ is adjacent to at
least two vertices in $S$, and hence there is such $B(S)$.

We assume that $B(S)$ is connected.
If $B(S)$ is disconnected, there is at least one connected component
$B'$ with the bipartition $S' \cup T'$, where $S' \subseteq S$ and $T'
\subseteq T$, that satisfies $|T'| \geq 2|S'|-1$.
Thus we can replace $B(S)$ with $B'$ in the following discussion.

Since $\deg_{B(S)}(v) = 2$ for $v \in T$, we define the multigraph
$B_{M}$ by replacing each vertex $v \in T$ and the two incident edges
with a single edge.
Hence, $B_{M}$ has a connected multigraph with the vertex set $S$, and
has $y$ edges that are corresponding to vertices of $T$.
Note that there are parallel edges between vertices $x$ and $y$ of
$B_{M}$ if and only if two or more vertices of $T$ are adjacent to $x$
and $y$ in $B(S)$.

\vspace{1em}
\noindent
\textbf{Claim 1.} $B_{M}$ has no pair of vertices such that there are
three or more parallel edges between them.
\begin{proof}[Proof of Claim~1.]
  If there are three parallel edges between $x$ and $y$ in $B_{M}$,
  then $B(S)$ has $K_{2,3}$ as a subgraph.
  Since $G$ is an outerplanar graph, it is impossible by
  Theorem~\ref{thm:charact}.
\end{proof}

\vspace{1em}
\noindent
\textbf{Claim 2.} $B_{M}$ has a block $F$ of $x'$ vertices and $y'$
edges such that $y' \geq x'+1$.
\begin{proof}[Proof of Claim~2.]
  If $B_{M}$ has no cut-vertex, then $B_{M}$ itself a block and we
  have $y \geq 2x-1 \geq x+1$.
  Assume that $B_{M}$ has $k \geq 2$ blocks $B_{1},B_{2},\dots,B_{k}$,
  and $B_{i}$ has $x_{i}$ vertices and $y_{i}$ edges for
  $i=1,2,\dots,k$.
  Suppose to the contrary that $y_{i} \leq x_{i}$ for all $i=1,2,\dots,k$.
  Thus $B_{M}$ has $y=y_{1}+y_{2}+\dots+y_{k}$ edges.
  Since $B_{M}$ has $k$ blocks, we obtain
  $y = y_{1}+y_{2}+ \dots +y_{k} \leq x_{1}+x_{2}+ \dots +x_{k} = x +
  k-1$.
  Since $y \geq 2x-1$, we obtain $x \leq k$.
  But a graph of $x \leq k$ vertices cannot have $k \geq 2$ blocks, so
  it is a contradiction.
\end{proof}

\vspace{1em}
Let $F$ be a block satisfying the condition of Claim~2.
If $x'=2$, then it has $y' \geq 3$ parallel edges between the two
vertices.
But it is impossible by Claim~1.
Hence $x' \geq 3$.
Since $F$ has no cut-vertex, $F$ contains a cycle.
Let $C$ be a shortest cycle of order at least 3 in $F$.
So $C$ has no chord.
If there is parallel edges between two consecutive vertices $v$ and
$w$ on $C$, there are three internally disjoint paths from $v$ to $w$,
two are the parallel edges and one is a path along the cycle $C$.
This implies $B(C)$ has a subdivision of $K_{2,3}$, and it is
impossible by Theorem~\ref{thm:charact}.
% If there is a chord in $C$, that is, there is an edge between vertices
% $v$ and $w$ such that $v$ and $w$ are not consecutive on $C$, then
% there are three internally disjoint paths from $v$ to $w$ in $F$, two
% are along the cycle $C$ and the one is edge $vw$.
% This also implies $B(C)$ has a subdivision of $K_{2,3}$, and it is
% impossible.
% Finally, we assume that $C$ has no parallel edges and no chords.
Then, we assume that $C$ has no parallel edges.
Since $y' \geq x'+1$, $F$ has at least one vertex $v$ not on $C$.
Since $F$ has no cut-vertex, there are two paths $P_{1}$ and $P_{2}$
such that
\begin{itemize}
\item $P_{1}$ is a path from $v$ to some vertex $v_{1}$ on $C$, and
  $V(P_{1}) \cap V(C) = \{v_{1}\}$,
\item $P_{2}$ is a path from $v$ to some vertex $v_{2}$ on $C$, and
  $V(P_{2}) \cap V(C) = \{v_{2}\}$, and
\item $v_{1} \neq v_{2}$ and $V(P_{1}) \cap V(P_{2}) = \{v\}$.
\end{itemize}
There are three internally disjoint paths from $v_{1}$ to $v_{2}$, two
are along the cycle $C$ and the one is obtained by concatenating
$P_{1}$ and $P_{2}$.
Again, we see that $B(C)$ has a subdivision of $K_{2,3}$, and it is
impossible.

From the above discussion, we show that $y \geq 2x-1$ is impossible.
Hence $y \leq 2x-2$ holds.
Since $n=x+y$, we obtain $n-x \leq 2x-2$ or $x \geq (n+2)/3$.
Therefore, any 2-dominating set $S$ has at least
$\lceil (n+2)/3 \rceil$ vertices, and the proof of
Theorem~\ref{thm:lower_bound} is complete.

\vspace{1em}
%%% 
We can construct infinite family of mops such that the equality of the
lower bound in Theorem~\ref{thm:lower_bound} holds.

For any $k \geq 1$, let $G_{k}$ be a graph that has the vertex set
$V(G_{k}) = \{v_{1},v_{2},\dots,v_{3k+1}\}$, and
\begin{displaymath}
  E(G_{k}) = \{v_{i}v_{i+1} \mid i = 2,3,\dots,3k\} \cup \{v_{1} v_{i} \mid i = 2,3,\dots,3k+1 \}.
\end{displaymath}
$G_{k}$ has a mop of $n=3k+1$ vertices, and $\deg_{G_{k}}(v_{1})=n-1$,
$\deg_{G_{k}}(v_{i})=2$ for $i \in \{2,3k+1\}$, and $\deg_{G_{k}}(v_{i})=3$
for $i \notin \{1,2,3k+1\}$.
An example of $G_{4}$ is in Fig.~\ref{fig:g_k}.

\begin{figure}[tb]
  \centering
  \includegraphics{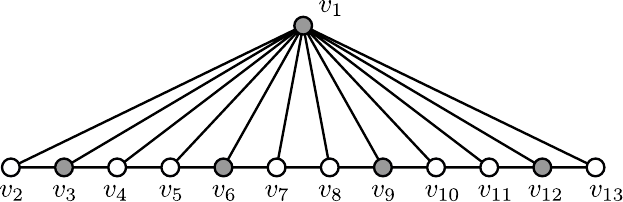}
  \caption{A outerplanar graph $G_{k}$ when $k=4$.
    The set of the gray vertices form an STDS of $G_{4}$.}
  \label{fig:g_k}
\end{figure}

Let $S=\{v_{1}\} \cup \{v_{3i} \mid i=1,2,\dots,k\}$.
Then $S$ is an STDS of $G_{k}$, and $|S|=\lceil (n+2)/3 \rceil=k+1$.
We show that $S$ is a minimum STDS of $G_{k}$.

Assume to the contrary that there is an STDS $S'$ of $k$ vertices.
For $j=1,2,\dots,k$, let $V_{j}=\{v_{3j-1},v_{3j},v_{3j+1}\}$.
Then $V(G)=V_{0} \cup V_{1} \cup \dots \cup V_{k}$, where
$V_{0}=\{v_{1}\}$.
Since $|S'|=k$, we have $V_{j} \cap S' = \emptyset$ for some $j$.
If $V_{j} \cap S' = \emptyset$ for some $j \geq 1$, then $v_{1} \in
S'$ since $S'$ is a total dominating set, but $v_{1}$ cannot totally
$S'$-defend $v_{3j}$ since $v_{3j} \in \epn(v_{1},S')$
(Theorem~\ref{prop:stds_ch}).
Hence $v_{1} \notin S'$.
Since $|S'|=k$ and $S'$ is a dominating set, $|S' \cap V_{j}|=1$ for
each $j=1,2,\dots,k$.
However, it implies that $S'$ cannot be a total dominating set.
Therefore, $\gamma_{st}(G_{k}) = k+1$ for any $k \geq 1$.

\section*{Acknowledgments}
\label{sec:acknowledgments}

This work was supported by JSPS KAKENHI Grant Number JP22K11898.



\begin{thebibliography}{10}

\bibitem{araki18:_secur}
  T.~Araki, H.~Miyazaki, Secure domination in proper interval graphs,
  Discrete Applied Mathematics 247 (2018) 70--76.

\bibitem{araki19:_secur}
  T.~Araki, R.~Yamanaka, Secure domination in cographs, Discrete
  Applied Mathematics 262 (2019) 179--184.

\bibitem{araki18}
  T.~Araki, I.~Yumoto, On the secure domination numbers of maximal
  outerplanar graphs, Discrete Applied Mathematics 236 (2018) 23--29.

\bibitem{benecke07:_secur}
  S.~Benecke, E.~J. Cockayne, C.~M. Mynhardt, Secure total domination
  in graphs, Utilitas Mathematicae 74 (2007) 247--259.

\bibitem{burger14}
  A.~P. Burger, A.~P. de~Villiers, J.~H. van Vuuren, A linear
  algorithm for secure domination in trees, Discrete Applied
  Mathematics 171~(10) (2014) 15--47.

\bibitem{burger16}
  A.~P. Burger, A.~P. de~Villiers, J.~H. van Vuuren, On minimum secure
  dominating sets of graphs, Quaestiones Mathematicae 39~(2) (2016)
  189--202.

\bibitem{burger08:_vertex}
  A.~P. Burger, M.~A. Henning, J.~H. van Vuuren, Vertex covers and
  secure domination in graphs, Quaestiones Mathematicae 31~(2) (2008)
  163--171.

\bibitem{campos13}
  C.~N. Campos, Y.~Wakabayashi, On dominating sets of maximal
  outerplanar graphs, Discrete Applied Mathematics 161 (2013)
  330--335.

\bibitem{chartrand11}
  G.~Chartrand, L.~Lesniak, P.~Zhang, Graphs \& Digraphs, 5th ed.,
  Chapman \& Hall/CRC, 2011.

\bibitem{cockayne08:_protec}
  E.~J. Cockayne, P.~J.~P. Grobler, W.~R. Gr\"{u}ndlingh, J.~Munganga,
  J.~H. van Vuuren, Protection of a graph, Utilitas Mathematicae 67
  (2005) 19--32.

\bibitem{dorfling16:_total2}
  M.~Dorfling, J.~H. Hattingh, E.~Jonck, Total domination in maximal
  outerplanar graphs {II}, Discrete Mathematics 339 (2016) 1180--1188.

\bibitem{duginov17:_secur}
  O.~Duginov, Secure total domination in graphs: Bounds and
  complexity, Discrete Applied Mathematics 222 (2017) 97--108.

\bibitem{jha19}
  A.~Jha, D.~Pradhan, S.~Banerjee, The secure domination problem in
  cographs, Discrete Applied Mathematics 145 (2019) 30--38.

\bibitem{klostermeyer08:_secur}
  W.~F. Klostermeyer, C.~M. Mynhardt, Secure domination and secure
  total domination in graphs, Dicussiones Mathematicae Graph Theory 28
  (2008) 267--284.

\bibitem{lemanska17:_total}
  M.~Lema\'{n}ska, R.~Zuazua, P.~\.{Z}yli\'{n}ski, Total dominating
  sets in maximal outerplanar graphs, Graphs and Combinatorics 33
  (2017) 991--998.

\bibitem{li17}
  Z.~Li, Z.~Shao, J.~Xu, On secure domination in trees, Quaestiones
  Mathematicae 40~(1) (2017) 1--12.

\bibitem{li16}
  Z.~Li, E.~Zhu, Z.~Shao, J.~Xu, On dominating sets of maximal
  outerplanar and planar graphs, Discrete Applied Mathematics 198
  (2016) 164--169.

\bibitem{martinez19}
  A.~C. Mart\'{i}nez, L.~P. Montejano,
  J.~A. Rodr\'{i}guez-Vel\'{a}zquez, On the secure total domination
  number of graphs, Symmetry 11~(9) (2019) 1165--1176.

\bibitem{matheson96:_domin}
  L.~R. Matheson, R.~E. Tarjan, Dominating sets in planar graphs,
  European Journal of Combinatorics 17 (1996) 565--568.

\bibitem{merouane15}
  H.~B. Merouane, M.~Chellali, On secure domination in graphs,
  Information Processing Letters 115 (2015) 786--790.

\bibitem{orourke87}
  J.~O'Rourke, Art gallery theorems and algorithms, Oxford University
  Press, New York, 1987.

\bibitem{pradhan17}
  D.~Pradhan, A.~Jha, On computing a minimum secure dominating set in
  block graphs, Journal of Combinatorial Optimization.

\bibitem{tokunaga13:_domin}
  S.~Tokunaga, Dominating sets of maximal outerplanar graphs, Discrete
  Applied Mathematics 161 (2013) 3097--3099.

\bibitem{zou19}
  Y.~H. Zou, J.~J. Liu, C.~C. Hsu, Y.~L. Wang, A simple algorithm for
  secure domination in proper interval graphs, Discrete Applied
  Mathematics 260 (2019) 289--293.

\end{thebibliography}
\end{document}